\documentclass[twoside,11pt]{amsart}
\usepackage{amsmath,amssymb,amsthm}

\newtheorem{theorem}{Theorem}[section]
\newtheorem{proposition}[theorem]{Proposition}
\newtheorem{lemma}[theorem]{Lemma}
\newtheorem{corollary}[theorem]{Corollary}

\theoremstyle{definition}
\newtheorem{definition}{Definition}
\newtheorem{convention}[definition]{Convention}

\theoremstyle{remark}
\newtheorem{remark}[theorem]{Remark}

\begin{document}

%
%
\newcommand{\ie}{i.e. }
\newcommand{\X}{\mathfrak{X}}
\newcommand{\C}{\mathcal{C}}
\newcommand{\W}{\mathcal{W}}
\newcommand{\F}{\mathcal{F}}
\newcommand{\T}{\mathcal{T}}
\newcommand{\LL}{\mathcal{L}}
\newcommand{\G}{\mathcal{G}}
\newcommand{\I}{\mathcal{I}}
\newcommand{\R}{\mathbb{R}}
\newcommand{\s}{\mathfrak{S}}
\newcommand{\n}{\nabla}
\newcommand{\nn}{\widetilde{\nabla}}
\newcommand{\tg}{\widetilde{g}}
\newcommand{\f}{\varphi}
\newcommand{\D}{{\rm d}}
\newcommand{\Id}{{\rm Id}}
\newcommand{\im}{{\rm im}}
\newcommand{\M}{(M,\f,\xi,\eta,g)}
\newcommand{\Span}{\mathrm{span}}
\newcommand{\lm}{\lambda}
\newcommand{\ta}{\theta}
\newcommand{\om}{\omega}
\newcommand{\sx}{\mathop{\mathfrak{S}}\limits_{x,y,z}}
\newcommand{\thmref}[1]{The\-o\-rem~\ref{#1}}
\newcommand{\propref}[1]{Pro\-po\-si\-ti\-on~\ref{#1}}
\newcommand{\secref}[1]{\S\ref{#1}}
\newcommand{\lemref}[1]{Lem\-ma~\ref{#1}}
\newcommand{\dfnref}[1]{De\-fi\-ni\-ti\-on~\ref{#1}}
\newcommand{\corref}[1]{Co\-rol\-la\-ry~\ref{#1}}
%
%
%
\title[A classification of the torsion tensors]
{A classification of the torsion tensors \\ on almost contact
manifolds with B-metric}

\author{Mancho Manev}
\address[M. Manev\footnote{the corresponding author}]{Department of Algebra and Geometry\\
Faculty of Mathematics and Informatics\\
Plovdiv University\\
236 Bulgaria Blvd, Plovdiv 4027\\
Bulgaria} %
\email{mmanev@uni-plovdiv.bg}

\author{Miroslava Ivanova}
\address[M. Ivanova]{Department of Informatics and Mathematics\\
Faculty of Economics\\
Trakia University\\
Student Campus, Stara Zagora 6000\\
Bulgaria} %
\email{mivanova@uni-sz.bg}

\subjclass[2000]{Primary 53C05; Secondary 53C15, 53C50. }

\keywords{torsion, almost contact manifold, B-metric, natural
connection}

\begin{abstract}
The space of the torsion (0,3)-tensors of the linear connections
on almost contact manifolds with B-metric is decomposed in 15
orthogonal and invariant subspaces with respect to the action of
the structure group. Three known connections, preserving the
structure, are characterized regarding this classification.
\end{abstract}

\maketitle

\section*{Introduction} 

The investigations of linear connections on almost contact
manifolds with B-metric take a central place in the study of the
differential geometry of these manifolds. The linear connections
preserving the metric are completely characterized by their
torsion tensors. In ac\-cord\-ance with our goals, it is important
to describe linear connections regarding the properties of their
torsion tensors with respect to the structures on the manifold.

Such a classification of the space of the torsion tensors is made
in \cite{GaMi87} in the case of almost complex manifolds with
Norden metric. These manifolds are the even-dimensional analogue
of the odd-dimensional almost contact manifolds.

The idea of decomposition of the space of the basic (0,3)-tensors,
generated by the covariant derivative of the fundamental tensor of
type $(1,1)$, is used by different authors in order to  obtain
classifications of manifolds with additional tensor structures.
For example, let us mention the classification of almost Hermitian
manifolds given in \cite{GrHe}, of almost complex manifolds with
Norden metric -- in \cite{GaBo}, of almost contact metric
manifolds -- in \cite{AlGa}, of almost contact manifolds with
B-metric -- in \cite{GaMiGr}, of Riemannian almost product
manifolds -- in \cite{Nav}, of Riemannian manifolds with traceless
almost product structure -- in \cite{StaGri}, of almost
paracontact metric manifolds -- in \cite{NakZam}, of almost
paracontact Riemannian manifolds of type $(n,n)$ -- in
\cite{ManSta01}.

The linear connections preserving the structure (also known as
natural connections) are particularly interesting in differential
geometry.
On an almost Hermitian manifold there exists a unique natural
connection $\n^C$ with a torsion $T$ which has the property
$T(J\cdot,J\cdot)=-T(\cdot,\cdot)$ with respect to the almost
complex structure $J$. This connection is known as the canonical
Hermitian connection or the Chern connection
\cite{Chern,Yano,YaKon}. An example of the natural Hermitian
connection is the first canonical connection of Lichnerowicz
$\n^L$ \cite{Li1,Li2}.
According to \cite{Gaud}, there exists a one-parameter family of
canonical Hermitian connections $\n^t=t\n^C+(1-t)\n^L$. The
connection $\n^t$ obtained for $t=-1$ is called the Bismut
connection or the KT-connection, which is characterized with a
totally skew-symmetric torsion \cite{Bis}. The latter connection
has applications in heterotic  string theory and in 2-dimensional
supersymmetric $\sigma$-models as well as in type II string theory
when the torsion 3-form is closed \cite{GHR,Stro,IvPapa,IvIv}. In
\cite{Fri-Iv2} and \cite{Fri-Iv} all almost Hermitian and almost
contact metric structures admitting a connection with totally
skew-symmetric torsion tensor are described.
Natural connections of canonical type are considered on the
Riemannian almost product manifolds in
\cite{Dobr11-1,Dobr11-2,MekDobr} and on the almost complex
manifolds with Norden metric in \cite{GaMi87,GaGrMi,Mek-JTech}.
The Tanaka-Webster connection on a contact metric manifold is
introduced (\cite{Tanno,Tan,Web}) in the context of CR-geometry. A
natural connection with minimal torsion on the quaternionic
contact structures, introduced in \cite{Biq}, is known as the
Biquard connection.
%


The goal of the present work is to describe the torsion space with
respect to the almost contact B-metric structure, which could be
used to study some natural connections on these manifolds.

This paper is organized as follows. In Sec.~\ref{sec:1}, we
present some necessary facts about the considered manifolds.
Sec.~\ref{sec:2} is devoted to the decomposition of the space of
torsion tensors on almost contact manifolds with B-metric. In
Sec.~\ref{sec:3}, we find the position of three known natural
connections in the obtained classification.

\begin{convention}\
\begin{enumerate}
    \item[(a)] We shall use $X$, $Y$, $Z$ to denote elements of
of the algebra $\X(M)$ on the smooth vector fields on $M$.
Moreover, $x$, $y$, $z$ will stand for arbitrary vectors in the
tangent space $T_pM$ of $M$ at an arbitrary point $p$ in $M$;
    \item[(b)]
    The notation $\sx$ means
the cyclic sum by the three arguments $x$, $y$, $z$. For example,
$\sx F(x,y,z)=F(x,y,z)+F(y,z,x)+F(z,x,y)$;
    \item[(c)]
    For the sake of brevity, we shall use the notation
    $\left\{A(x,y)\right\}_{[x\leftrightarrow
y]}$ for the difference $A(x,y)-A(y,x)$ and
$\left\{A(x,y)\right\}_{(x\leftrightarrow y)}$ for the sum
$A(x,y)+A(y,x)$, where $A$ is an arbitrary tensor. Similarly, we
use $\{A(x,y,z)\}_{[x\leftrightarrow y]}=A(x,y,z)-A(y,x,z)$ and
$\{A(x,y,z)\}_{(x\leftrightarrow y)}=A(x,y,z)+A(y,x,z)$ for any
tensor $A(x,y,z)$;
    \item[(d)] We shall use double subscripts separated by the symbol $\slash$.
    The former and latter subscripts regarding this symbol
correspond to the upper and down signs plus or minus in the same
equality, respectively.
For example, the notation %
$\F_{8/9}: F(x,y,z)=F(x,y,\xi)\eta(z)+F(x,z,\xi)\eta(y)$,
$F(x,y,\xi)=\pm F(y,x,\xi)=F(\f x,\f y,\xi)$ means %
$\F_{8}: F(x,y,z)=F(x,y,\xi)\eta(z)+F(x,z,\xi)\eta(y)$,
$F(x,y,\xi)= F(y,x,\xi)=F(\f x,\f y,\xi)$ %
and $\F_{9}: F(x,y,z)=F(x,y,\xi)\eta(z)+F(x,z,\xi)\eta(y)$,
$F(x,y,\xi)=- F(y,x,\xi)=F(\f x,\f y,\xi)$. Similarly,
$\W_{1,1/1,2}=\bigl\{T\in \W_1^-\ \vert\
L_{1,1}(T)\allowbreak{}=\mp T\bigr\}$ means $\W_{1,1}=\bigl\{T\in
\W_1^-\ \vert\ $ $L_{1,1}(T)=-T\bigr\}$ and $\W_{1,2}=\bigl\{T\in
\W_1^-\ \vert\ \allowbreak{}L_{1,1}(T)=T\bigr\}$.
\end{enumerate}
\end{convention}

\section{Almost Contact Manifolds with B-Metric}\label{sec:1}

Let $(M,\f,\xi,\eta,g)$ be an almost contact manifold with
B-met\-ric or an \emph{almost contact B-metric manifold}, \ie $M$
is a $(2n+1)$-dimensional differentiable manifold with an almost
contact structure $(\f,\xi,\eta)$ consisting of an endomorphism
$\f$ of the tangent bundle, a vector field $\xi$, its dual 1-form
$\eta$ as well as $M$ is equipped with a pseudo-Rie\-mannian
metric $g$  of signature $(n,n+1)$, such that the following
algebraic relations are satisfied: $\f\xi = 0$, $\f^2 = -\Id +
\eta \otimes \xi$, $\eta\circ\f=0$, $\eta(\xi)=1$, $g(\f X, \f Y )
= - g(X, Y ) + \eta(X)\eta(Y)$ \cite{GaMiGr}.

%

The associated metric $\tg$ of $g$ on $M$ is defined by \(
\tg(X,Y)=g(X,\f Y)+\eta(X)\eta(Y). \) The manifold
$(M,\f,\xi,\eta,\tg)$ is also an almost contact B-metric manifold.
Both metrics $g$ and $\tg$ are necessarily of signature $(n,n+1)$.
The Levi-Civita connection of $g$ and $\tg$ will be denoted by
$\n$ and $\nn$, respectively.

Let us remark that the $2n$-dimensional contact distribution
$H=\ker(\eta)$, generated by the contact 1-form $\eta$, can be
considered as the horizontal distribution of the
sub-Rie\-mann\-ian manifold $M$. Then $H$ is endowed with an
almost complex structure determined as $\f|_H$ -- the restriction
of $\f$ on $H$, as well as a Norden metric $g|_H$, i.e.
$g|_H(\f|_H\cdot,\f|_H\cdot)=-g|_H(\cdot,\cdot)$. Moreover, $H$
can be considered as a $n$-dimensional complex Riemannian
man\-i\-fold with a complex Riemannian metric
$g^{\mathbb{C}}=g|_H+i\tg|_H$ \cite{GaIv}.

The structure group of $\M$ is $\G\times\I$, where $\I$ is the
identity on $\Span(\xi)$ and $\G=\mathcal{GL}(n;\mathbb{C})\cap
\mathcal{O}(n,n)$, i.e. it consists of the real square matrices of
order $2n+1$ of the following type
\[
\left(%
\begin{array}{r|c|c}
  A & B & \vartheta^T\\ \hline
  -B & A & \vartheta^T\\ \hline
  \vartheta & \vartheta & 1 \\
\end{array}%
\right),\qquad %
\begin{array}{l}
  A^TA-B^TB=I_n,\\%
  B^TA+A^TB=O_n,
\end{array}%
\quad A, B\in \mathcal{GL}(n;\mathbb{R}),
\]
where $\vartheta$ and its transpose $\vartheta^T$ are the zero row
$n$-vector and the zero column $n$-vector; $I_n$ and $O_n$ are the
unit matrix and the zero matrix of size $n$, respectively.

A classification of almost contact manifolds with B-metric is
given in \cite{GaMiGr}. This classification, consisting of eleven
basic classes $\F_1$, $\F_2$, $\dots$, $\F_{11}$,  is made with
respect to the tensor  $F$ of type (0,3) defined by
$F(x,y,z)=g\bigl( \left( \nabla_x \f \right)y,z\bigr)$
and having the following properties
\(F(x,y,z)=F(x,z,y)=F(x,\f y,\f z)+\eta(y)F(x,\xi,z)
+\eta(z)F(x,y,\xi). \)

If $\left\{e_i;\xi\right\}$ $(i=1,2,\dots,2n)$ is a basis of
$T_pM$ and $\left(g^{ij}\right)$ is the inverse matrix of
$\left(g_{ij}\right)$, then the following 1-forms are associated
with $F$: $\theta(z)=g^{ij}F(e_i,e_j,z)$,
$\theta^*(z)=g^{ij}F(e_i,\f e_j,z)$, $\omega(z)=F(\xi,\xi,z)$.

Further we use the following characteristic conditions of the
basic classes:
\begin{equation}\label{Fi}
\begin{array}{rl}
\F_{1}: &F(x,y,z)=\frac{1}{2n}\bigl\{g(x,\f y)\ta(\f z)+g(\f x,\f
y)\ta(\f^2 z)
\bigr\}_{(y\leftrightarrow z)};\\[4pt]
\F_{2}: &F(\xi,y,z)=F(x,\xi,z)=0,\quad
              \sx F(x,y,\f z)=0,\quad \ta=0;\\[4pt]
\F_{3}: &F(\xi,y,z)=F(x,\xi,z)=0,\quad
              \sx F(x,y,z)=0;\\[4pt]
\F_{4}: &F(x,y,z)=-\frac{1}{2n}\ta(\xi)\bigl\{g(\f x,\f y)\eta(z)+g(\f x,\f z)\eta(y)\bigr\};\\[4pt]
\F_{5}: &F(x,y,z)=-\frac{1}{2n}\ta^*(\xi)\bigl\{g( x,\f y)\eta(z)+g(x,\f z)\eta(y)\bigr\};\\[4pt]
\F_{6/7}: &F(x,y,z)=F(x,y,\xi)\eta(z)+F(x,z,\xi)\eta(y),\quad \\[4pt]
                &F(x,y,\xi)=\pm F(y,x,\xi)=-F(\f x,\f y,\xi),\quad \ta=\ta^*=0; \\[4pt]
\F_{8/9}: &F(x,y,z)=F(x,y,\xi)\eta(z)+F(x,z,\xi)\eta(y),\\[4pt]
                &F(x,y,\xi)=\pm F(y,x,\xi)=F(\f x,\f y,\xi); \\[4pt]
\F_{10}: &F(x,y,z)=F(\xi,\f y,\f z)\eta(x); \\[4pt]
\F_{11}:
&F(x,y,z)=\eta(x)\left\{\eta(y)\om(z)+\eta(z)\om(y)\right\}.
\end{array}
\end{equation}

The intersection of the basic classes is the special class $\F_0$
determined by the condition $F(x,y,z)=0$. Hence $\F_0$ is the
class of almost contact B-metric manifolds with $\n$-parallel
structures, i.e. $\n\f=\n\xi=\n\eta=\n g=\n \tilde{g}=0$.

\subsection{Associated tensor of the Nijenhuis
tensor}\label{sec:1a}\

The Nijenhuis tensor of the contact structure is defined by $N =
[\f, \f]+ \D{\eta}\otimes\xi, $ where $[\f, \f](x, y)=\left[\f
x,\f y\right]+\f^2\left[x,y\right]-\f\left[\f
x,y\right]-\f\left[x,\f y\right]$ %
is the Nijenhuis torsion of $\f$ and $\D{\eta}$ is the exterior
derivative of the 1-form $\eta$.

By analogy with the skew-symmetric Lie bracket $[x,y]=\n_x y-\n_y
x$, let us consider the symmetric bracket $\{x,y\}=\n_x y+\n_y x$.
Then we introduce the symmetric tensor $\{\f, \f\}(x, y)=\{\f x,\f
y\}+\f^2\{x,y\}-\f\{\f x,y\}-\f\{x,\f
 y\}$.
%
Additionally, we use the Lie derivative of the metric $g$ along
$\xi$, \ie $\left(\LL_{\xi}g\right)(x,y)
 =\left(\n_x\eta\right)y+\left(\n_y\eta\right)x,$ as an alternative
 of
 $\D{\eta}(x,y)=\left(\n_x\eta\right)y-\left(\n_y\eta\right)x.$
Then, we define
an associated tensor $\widehat{N}$ with $N$ 
by:
\begin{equation}\label{S}
\widehat{N} =\{\f,\f\}+\left(\LL_{\xi}g\right)\otimes\xi.
\end{equation}

It is well known that the Nijenhuis tensor $N$ is determined by
covariant derivatives of $\f$ and $\eta$ with respect to $\n$ as
follows:
\begin{equation}\label{N}
N(x,y)=\left\{\left(\n_{\f
x}\f\right)y-\f\left(\n_{x}\f\right)y+\left(\n_{x}\eta\right)y\cdot\xi\right\}_{[x\leftrightarrow
y]}.
\end{equation}

\begin{proposition}\label{prop-Nhat=nabla}
The tensor $\widehat{N}$ has the following form in terms of
  $\n\f$ and $\n\eta$:
\begin{equation}\label{Sn}
\widehat{N}(x,y)=\left\{\left(\n_{\f
x}\f\right)y-\f\left(\n_{x}\f\right)y+\left(\n_{x}\eta\right)y\cdot\xi
\right\}_{(x\leftrightarrow y)}.
\end{equation}
\end{proposition}
\begin{proof}
We obtain immediately
\[
\begin{array}{l}
\widehat{N}(x,y)
=\{\f,\f\}(x,y)+\left(\LL_{\xi}g\right)(x,y)\cdot\xi=\{\f x,\f
y\}+\f^2\{x,y\}-\f\{\f x,y\}-\f\{x,\f y\}
\\[4pt]
\phantom{\widehat{N}(x,y)=}
+\left(\n_x\eta\right)y\cdot\xi+\left(\n_y\eta\right)x\cdot\xi=\n_{\f
x}\f y+\n_{\f y}\f x +\f^2\n_{x}y+\f^2\n_{y}x-\f\n_{\f
x}y\\[4pt]
\phantom{\widehat{N}(x,y)=} %
-\f\n_{y}\f x-\f\n_{x}\f y-\f\n_{\f
y}x+\left(\n_{x}\eta\right)y\cdot\xi
+\left(\n_{y}\eta\right)x\cdot\xi\\[4pt]
\phantom{\widehat{N}(x,y)} %
=\left(\n_{\f x}\f\right)y+\left(\n_{\f
y}\f\right)x-\f\left(\n_{x}\f\right)y-\f\left(\n_{y}\f\right)x
+\left(\n_{x}\eta\right)y\cdot\xi+\left(\n_{y}\eta\right)x\cdot\xi,
\end{array}
\]
which completes the proof.
\end{proof}

It is known that the class of the normal almost contact B-metric
manifolds, \ie $N=0$, is
$\F_1\oplus\F_2\oplus\F_4\oplus\F_5\oplus\F_6$.

\begin{proposition}\label{prop-Nhat=0}
The class of the almost contact B-metric manifolds with
$\widehat{N}=0$ is $\F_3\oplus\F_7$.
\end{proposition}
\begin{proof}
By virtue of \eqref{Sn} and the form of $F(x,y,z)=g\bigl( \left(
\nabla_x \f \right)y,z\bigr)$ in \eqref{Fi}, we establish that
$\widehat{N}$ has the following form on $M=\M$ belonging to $\F_i$
$(i=1,2, \dots, 11)$, respectively:
\[
\begin{array}{ll}
    \widehat{N}(x,y)=\frac{2}{n}\bigl\{g(\f x,\f y)\f \Theta+g(x,\f
y)\Theta\bigr\}, \quad &M\in\F_1;\\[4pt]
    \widehat{N}(x,y)=2\left\{\left(\n_{\f
x}\f\right)y-\f\left(\n_{x}\f\right)y\right\}, \quad
&M\in\F_2;\\[4pt]
    \widehat{N}(x,y)=0,
\quad &M\in\F_3\oplus \F_7;\\[4pt]
    \widehat{N}(x,y)=\frac{2}{n}\ta(\xi)g(x,\f y) \cdot \xi,
\quad &M\in\F_4;\\[4pt]
\end{array}
\]
\[
\begin{array}{ll}
    \widehat{N}(x,y)=-\frac{2}{n}\ta^*(\xi)g(\f x,\f y) \cdot \xi,
\quad &M\in\F_5;\\[4pt]
    \widehat{N}(x,y)=4\left(\n_{x}\eta\right)y \cdot \xi,
\quad &M\in\F_6;\\[4pt]
    \widehat{N}(x,y)=-2\left\{\eta(x) \n_{y}\xi+\eta(y)
\n_{x}\xi\right\}, \quad &M\in\F_8\oplus\F_9;\\[4pt]
    \widehat{N}(x,y)=-\left\{\eta(x)\f \left(
\n_{\xi}\f\right)y+\eta(y)\f \left( \n_{\xi}\f\right)x\right\},
\quad &M\in\F_{10};\\[4pt]
    \widehat{N}(x,y)=-2\eta(x)\eta(y)\f \Omega
+\left\{\eta(x)\om(\f y)+\eta(y)\om(\f
x)\right\} \cdot \xi, \quad &M\in\F_{11},
\end{array}
\]
where $\ta(z)=g(\Theta,z)$ and $\om(z)=g(\Omega,z)$. Then the
truthfulness of the statement follows.
\end{proof}

\section{A Decomposition of the Space of Torsion Tensors}\label{sec:2}

The object of our considerations are the linear connections with
torsion. Thus, we have to study the properties of the torsion
tensors with respect to the contact structure and the B-metric.

If $T$ is the torsion tensor of $D$, \ie $T(x,y)=D_x y-D_y x-[x,
y]$, then the corre\-sponding tensor of type (0,3) is determined
by $T(x,y,z)=g(T(x,y),z)$.

Let us consider $T_pM$ at arbitrary $p\in M$ as a
$(2n+1)$-dimension\-al vector space with almost contact B-metric
structure $(V,\f,\xi,\eta,g)$. Moreover, let $\T$ be the vector
space of all tensors $T$ of type (0,3) over $V$ having
skew-symmetry by the first two arguments, \ie \[
\T=\left\{T(x,y,z)\in\R,\; x,y,z\in V \; \vert\;\;
T(x,y,z)=-T(y,x,z)\right\}. \]

The metric $g$ induces an inner product
$\langle\cdot,\cdot\rangle$ on $\T$ defined by $\langle
T_1,T_2\rangle=g^{iq}g^{jr}g^{ks}\allowbreak{}
T_1(e_i,e_j,e_k)\allowbreak{}T_2(e_q,e_r,e_s)$ for any
$T_{1,2}\in\T$ and a basis $\left\{e_i\right\}$
$(i=1,2,\dots,2n+1)$ of $V$.

The standard representation of the structure group $\G\times\I$
 in $V$ induces a natural representation $\lm$ of $\G\times\I$ in
 $\T$ as follows
$ \left((\lm
a)T\right)(x,y,z)=T\left(a^{-1}x,a^{-1}y,a^{-1}z\right) $ for any
$a\in \G\times\I$ and $T\in\T$, so that \( \langle(\lm a)T_1,(\lm
a)T_2\rangle=\langle T_1,T_2\rangle\), \(T_1,T_2\in \T. \)

The decomposition $x=-\f^2x+\eta(x)\xi$ generates the projectors
$h$ and $v$ on $V$ determined by $h(x)=-\f^2x$ and
$v(x)=\eta(x)\xi$ and having the properties $h\circ h =h$, $v\circ
v=v$, $h\circ v=\allowbreak{}v\circ h=0$. Therefore, we have the
orthogonal decomposition $V=h(V)\oplus v(V)$.

Bearing in mind these projectors on $V$, we construct a partial
decomposition of $\T$ as follows.

At first, we define the operator $p_1:\ \T\rightarrow\T$ by
\[
p_1(T)(x,y,z)=-T(\f^2x,\f^2y,\f^2z),\quad T\in\T.
\]
It is easy to check the following
\begin{lemma}\label{lem-p1}
The operator $p_1$ has the following properties$:$\\[4pt]
\begin{tabular}{rlrl}
    $($i$)$& $p_1\circ p_1 = p_1;\qquad\qquad\qquad$&
    $($ii$)$& $\langle p_1(T_1),T_2\rangle=\langle
    T_1,p_1(T_2)\rangle,\quad T_1, T_2 \in\T;$\\
    $($iii$)$& $p_1\circ (\lm a)=(\lm a)\circ p_1$.&&
\end{tabular}
\end{lemma}
According to \lemref{lem-p1} we have the following orthogonal
decomposition of $\T$ by the image and the kernel of $p_1$:
\[
\begin{split}
\W_1=\im(p_1)=\left\{T\in\T\ \vert\ p_1(T)=T\right\},\quad
\W_1^\bot=\ker(p_1)=\left\{T\in\T\ \vert\ p_1(T)=0\right\}.
\end{split}
\]

Further, we consider the operator $p_2:\
\W_1^\bot\rightarrow\W_1^\bot$, defined by
\[
p_2(T)(x,y,z)=\eta(z)T(\f^2 x, \f^2 y, \xi),\quad T\in\W_1^\bot.
\]
We obtain immediately the truthfulness of the following
\begin{lemma}\label{lem-p2}
The operator $p_2$ has the following properties$:$\\[4pt]
    \begin{tabular}{rlrl}
    $($i$)$& $p_2\circ p_2 = p_2;\qquad\qquad\qquad$&
    $($ii$)$& $\langle p_2(T_1),T_2\rangle=\langle
    T_1,p_2(T_2)\rangle,\quad T_1, T_2 \in\W_1^\bot;$\\
    $($iii$)$& $p_2\circ (\lm a)=(\lm a)\circ p_2$.&&
\end{tabular}
\end{lemma}
Then, bearing in mind  \lemref{lem-p2}, we obtain
\[
\begin{split}
\W_2=\im(p_2)=\left\{T\in\W_1^\bot\ \vert\ p_2(T)=T\right\},\quad
\W_2^\bot=\ker(p_2)=\left\{T\in\W_1^\bot\ \vert\ p_2(T)=0\right\}.
\end{split}
\]

Finally, we consider the operator $p_3:\
\W_2^\bot\rightarrow\W_2^\bot$ defined by
\[
p_3(T)(x,y,z)=\eta(x)T(\xi,\f^2 y, \f^2 z)+\eta(y)T(\f^2
x,\xi,\f^2 z),\quad T\in\W_2^\bot
\]
and we get the following
\begin{lemma}\label{lem-p3}
The operator $p_3$ has the following properties$:$\\[4pt]
    \begin{tabular}{rlrl}
    $($i$)$& $p_3\circ p_3 = p_3;\qquad \qquad \qquad$&
    $($ii$)$& $\langle p_3(T_1),T_2\rangle=\langle
    T_1,p_3(T_2)\rangle,\quad T_1, T_2 \in\W_2^\bot;$\\[4pt]
    $($iii$)$& $p_3\circ (\lm a)=(\lm a)\circ p_3$.&&
\end{tabular}
\end{lemma}
By virtue of \lemref{lem-p3}, we have
\[
\begin{split}
\W_3=\im(p_3)=\left\{T\in\W_2^\bot\ \vert\ p_3(T)=T\right\},\quad
\W_4=\ker(p_3)=\left\{T\in\W_2^\bot\ \vert\ p_3(T)=0\right\}.
\end{split}
\]

From \lemref{lem-p1}, \lemref{lem-p2} and \lemref{lem-p3} we have
immediately
\begin{theorem}\label{thm-W1234}
The decomposition $ \T=\W_1\oplus\W_2\oplus\W_3\oplus\W_4 $ is
orthogonal and invariant under the action of $\G\times\I$. The
subspaces $\W_i$ $(i=1,2,3,4)$ are determined by
\begin{equation}\label{W1234}
\begin{split}
\W_1:\quad &T(x,y,z)=-T(\f^2 x, \f^2 y, \f^2 z),\quad
\W_2:\quad T(x,y,z)=\eta(z)T(\f^2 x, \f^2 y, \xi),\\[4pt]
\W_3:\quad &T(x,y,z)=\eta(x)T(\xi,\f^2 y, \f^2 z)+\eta(y)T(\f^2 x,\xi,\f^2 z),\\[4pt]
\W_4:\quad &T(x,y,z)=-\eta(z)\left\{\eta(y)T(\f^2 x, \xi,\xi)+\eta(x)T(\xi,\f^2 y,\xi)\right\}\\[4pt]
\end{split}
\end{equation}
for arbitrary vectors $x, y,z \in V$.
\end{theorem}

\begin{corollary}\label{cor-W1234}
The subspaces $\W_i$ $(i=1,2,3,4)$ are characterized as follows$:$
\[
\begin{split}
&\W_1=\left\{T\in\T\ \vert\ T(v(x),y,z)=T(x,y,v(z))=0\right\},\\[4pt]
&\W_2=\left\{T\in\T\ \vert\ T(v(x),y,z)=T(x,y,h(z))=0\right\},\\[4pt]
&\W_3=\left\{T\in\T\ \vert\ T(x,y,v(z))=T(h(x),h(y),z)=0\right\},\\[4pt]
&\W_4=\left\{T\in\T\ \vert\ T(x,y,h(z))=T(h(x),h(y),z)=0\right\},\\[4pt]
\end{split}
\]
where $x, y,z \in V$.
\end{corollary}

The torsion forms associated with $T\in\T$ are defined as follows:
\begin{equation}\label{t}
\begin{array}{c}
t(x)=g^{ij}T(x,e_i,e_j),\qquad t^*(x)=g^{ij}T(x,e_i,\f e_j),\qquad
\hat{t}(x)=T(x,\xi,\xi)
\end{array}
\end{equation}
regarding the basis $\left\{e_i;\xi\right\}$ $(i=1,2,\dots,2n)$ of
$V$. Obviously, $\hat{t}(\xi)=0$ is always valid.

According to \corref{cor-W1234}, \eqref{W1234} and \eqref{t} we
obtain the following
\begin{corollary}\label{cor-t-W1}
The torsion forms of $T$ have the following properties in each of
the
sub\-spaces $\W_{i}$ $(i=1,2,3,4)$$:$\\[4pt]
\begin{tabular}{rlrl}
    $($i$)$& If $T\in\W_{1}$, then $t\circ v=t^*\circ
    v=\hat{t}=0;$ \qquad &
    $($ii$)$& If $T\in\W_{2}$, then $t=t^*=\hat{t}=0;$\\[4pt]
    $($iii$)$& If $T\in\W_{3}$, then $t\circ h=t^*\circ
    h=\hat{t}=0;$ \qquad &
    $($iv$)$& If $T\in\W_{4}$, then $t=t^*=0$.
\end{tabular}
\end{corollary}

Further we continue the decomposition of the subspaces $\W_i$
$(i=1,2,3,4)$ of $\T$.

\subsection{The subspace $\W_1$}
Since the endomorphism $\f$ induces an almost complex structure on
$H=\ker(\eta)$ (which is the orthogonal complement $\{\xi\}^\bot$
of the subspace $\Span(\xi)$) and the restriction of $g$ on $H$ is
a Norden metric (because the almost complex structure causes an
anti-isometry on $H$), then the decomposition of $\W_1$ is made as
the decomposition of the space of the torsion tensors on an almost
complex manifold with Norden metric known from \cite{GaMi87}.

Let us consider the linear operator $L_{1,0}:\ \W_1\rightarrow
\W_1$ defined by
\[
L_{1,0}(T)(x,y,z)=-T(\f x,\f y,\f^2 z).
\]
Then, it follows immediately
\begin{lemma}\label{lem-L10}
The operator $L_{1,0}$ is an involutive isometry on $\W_1$ and it
is invariant with respect to the group $\G\times\I$, \ie
\[
\begin{array}{c}
L_{1,0}\circ L_{1,0}=\Id_{\W_1}, \quad \langle
L_{1,0}(T_1),L_{1,0}(T_2)\rangle=\langle T_1,T_2 \rangle, \quad
L_{1,0}((\lm a)T)=(\lm a)(L_{1,0}(T)),
\end{array}
\]
where $T_1,T_2\in\W_1$, $a\in \G\times\I$.
\end{lemma}

Therefore, $L_{1,0}$ has two eigenvalues $+1$ and $-1$, and the
corresponding eigen\-spaces
\[
\W_1^+=\left\{T\in \W_1\ \vert\ L_{1,0}(T)= T\right\},\qquad
\W_1^-=\left\{T\in \W_1\ \vert\ L_{1,0}(T)=- T\right\}
\]
are invariant orthogonal subspaces of $\W_1$.

In order to decompose $\W_1^-$, we consider the linear operator
$L_{1,1}:\ \W_1^-\rightarrow \W_1^-$ defin\-ed by
\[
L_{1,1}(T)(x,y,z)=-T(\f x,\f^2 y,\f z).
\]
Let us denote the eigenspaces $\W_{1,1/1,2}=\left\{T\in \W_1^-\
\vert\ L_{1,1}(T)=\mp T\right\}$.
We have
\begin{lemma}\label{lem-L11}
The operator $L_{1,1}$ is an involutive isometry on $\W_1$ and it
is invariant with respect to $\G\times\I$.
\end{lemma}

According to the latter lemma, the eigenspaces $\W_{1,1}$ and
$\W_{1,2}$ are invariant and orthogonal.

To decompose $\W_1^+$, we define the linear operator $L_{1,2}:\
\W_1^+\rightarrow \W_1^+$ as follows:
\[
\begin{array}{l}
L_{1,2}(T)(x,y,z)=-\frac{1}{2}\left\{T(\f^2 z,\f^2 x,\f^2 y)
+T(\f^2 z,\f x,\f y)\right\}_{[x\leftrightarrow y]}.
\end{array}
\]
%
\begin{lemma}\label{lem-L12}
The operator $L_{1,2}$ is an involutive isometry on $\W_1^+$ and
it is invariant with respect to $\G\times\I$.
\end{lemma}

Thus, the eigenspaces $\W_{1,3/1,4}=\left\{T\in \W_1^+\ \vert\
L_{1,2}(T)=\pm T\right\}$ are invariant and orthogonal.

Using \lemref{lem-L10}, \lemref{lem-L11} and \lemref{lem-L12}, we
get the following
\begin{theorem}\label{thm-T1k}
The decomposition $
\W_1=\W_{1,1}\oplus\W_{1,2}\oplus\W_{1,3}\oplus\W_{1,4} $ is
orthogonal and invariant with respect to the structure group.
\end{theorem}

Bearing in mind the definition of the subspaces $\W_{1,i}$
$(i=1,2,3,4)$, we obtain
\begin{proposition}\label{prop-T1k}
The subspaces $\W_{1,i}$ $(i=1,2,3,4)$ of $\W_1$ are determined
by$:$
\[
\begin{split}
\W_{1,1}:\quad
    &T(\xi,y,z)=T(x,y,\xi)=0,\quad 
    T(x,y,z)=-T(\f x,\f y,z)=-T(x,\f y,\f z);\\[4pt]
%
\W_{1,2}:\quad &T(\xi,y,z)=T(x,y,\xi)=0,\quad 
            T(x,y,z)=-T(\f x,\f y,z)=T(\f x,y,\f z);\\[4pt]
\W_{1,3}:\quad &T(\xi,y,z)=T(x,y,\xi)=0,\quad 
            T(x,y,z)-T(\f x,\f y,z)=\sx  T(x,y,z)=0;\\[4pt]
\W_{1,4}:\quad &T(\xi,y,z)=T(x,y,\xi)=0,\quad 
            T(x,y,z)-T(\f x,\f y,z)=\sx T(\f x,y,z)=0.
\end{split}
\]
\end{proposition}


Using \corref{cor-t-W1} (i), \propref{prop-T1k} and \eqref{t}, we
obtain
\begin{corollary}\label{cor-t-T1i}
The torsion forms $t$ and $t^*$ of $T$ have the following
properties in the subspaces $\W_{1,i}$ $(i=1,2,3,4)$$:$\\[4pt]
\begin{tabular}{rlrl}
    $($i$)$& If $T\in\W_{1,1}$, then $t=-t^*\circ\f$, $t\circ\f=t^*;$\qquad &
    $($ii$)$& If $T\in\W_{1,2}$, then $t=t^*=0;$\\[4pt]
    $($iii$)$& If $T\in\W_{1,3}$, then $t=t^*\circ\f$, $t\circ\f=-t^*;$\qquad &
    $($iv$)$& If $T\in\W_{1,4}$, then $t=t^*=0$.
\end{tabular}
\end{corollary}

Let us remark that each of the subspaces $\W_{1,1}$ and $\W_{1,3}$
can be additionally decomposed to a couple of subspaces --- one of
zero traces $(t,\ t^*)$ and one of non-zero traces $(t,\ t^*)$,
\ie
\begin{equation}\label{W11W13}
\begin{array}{c}
\W_{1,1}=\W_{1,1,1}\oplus \W_{1,1,2}, \qquad
\W_{1,3}=\W_{1,3,1}\oplus \W_{1,3,2},
\end{array}
\end{equation}
where
\begin{equation*}\label{W111W112}
\begin{split}
&\W_{1,1,1}=\left\{T\in \W_{1,1}\ \vert\ t\neq 0\right\}, \qquad
\W_{1,3,1}=\left\{T\in \W_{1,3}\ \vert\ t\neq 0\right\},\\[4pt]
&\W_{1,1,2}=\left\{T\in \W_{1,1}\ \vert\ t=0\right\},\qquad
\W_{1,3,2}=\left\{T\in \W_{1,3}\ \vert\ t=0\right\}.
\end{split}
\end{equation*}

\begin{proposition}\label{prop-p_1i}
Let $T\in\T$ and $p_{1,i}$ $(i=1,2,3,4)$ be the projection
operators of $\T$ in $\W_{1,i}$, generated by the decomposition
above. Then we have
\[
\begin{array}{l}
p_{1,1/1,2}(T)(x,y,z)=-\frac{1}{4}\left\{T(\f^2 x,\f^2 y,\f^2 z)-T(\f x,\f y,\f^2 z)\right.
                            \left.\mp T(\f x,\f^2 y,\f z)\right.\\[4pt]
\phantom{p_{1,1/1,2}(T)(x,y,z)=-\frac{1}{4}\left\{\right.}\left.\mp T(\f^2 x,\f y,\f z)\right\};\\[4pt]
p_{1,3/1,4}(T)(x,y,z)=-\frac{1}{4} \left\{T(\f^2 x,\f^2 y,\f^2
z)+T(\f x,\f y,\f^2 z) \right\}\pm\frac{1}{8} \left\{T(\f^2
z,\f^2 x,\f^2 y) \right.\\[4pt]
\phantom{p_{1,3/1,4}(T)(x,y,z)=}\left.+ T(\f^2 z,\f x,\f y)  +
T(\f z,\f x,\f^2 y)- T(\f z,\f^2 x,\f
y)\right\}_{[x\leftrightarrow y]}.
\end{array}
\]
\end{proposition}
\begin{proof}
Let us show the calculations about $p_{1,1}$ for example, using
\cite{GaMi87}. \lemref{lem-L10} implies that the tensor
$\frac{1}{2}\left\{T-L_{1,0}(T)\right\}$ is the projection of
$T\in\W_1$ in $\W_1^-=\W_{1,1}\oplus\W_{1,2}$. Using
\lemref{lem-L11}, we find the expression of $p_{1,1}$ in terms of
the operators $L_{1,0}$ and $L_{1,1}$ for $T\in\W_1$, namely
\[
\begin{array}{l}
p_{1,1}(T)=\frac{1}{4}\left\{T-L_{1,0}(T)-L_{1,1}(T)+L_{1,1}\circ
L_{1,0}(T)\right\}, \end{array}
\]
which  implies the stated expression of $p_{1,1}$, taking into
account that $T\in \W_1$ is the image of $T\in\T$ by $p_1$. In a
similar way we prove the expressions for the other projectors
under consideration.

We verify that $p_{1,i}\circ p_{1,i}=p_{1,i}$ and $\sum_i
p_{1,i}=\Id_{\W_{1}}$ for $i=1,2,3,4$.
\end{proof}

\subsection{The subspace $\W_2$} Following the demonstrated
procedure for $\W_1$, we continue the decomposition of the other
main subspaces of $\T$ with respect to the almost contact B-metric
structure.

\begin{lemma}
The operator $L_{2,0}$, defined by \(
L_{2,0}(T)(x,y,z)=\eta(z)T(\f x,\f y,\xi), \) is an involutive
isometry on $\W_2$ and invariant with respect to $\G\times\I$.
\end{lemma}

Hence, the corresponding eigenspaces $ \W_{2,1/2,2}=\left\{T\in
\W_2\ \vert\ L_{2,0}(T)=\mp T\right\} $
 are
invariant and orthogonal. Therefore, we have
\begin{theorem}\label{thm-T2k}
The decomposition $ \W_2=\W_{2,1}\oplus\W_{2,2} $ is orthogonal
and invariant with respect to the structure group.
\end{theorem}

\begin{proposition}\label{prop-T2k}
The subspaces of $\W_2$ are determined by$:$
\[
\W_{2,1/2,2}:\quad T(x,y,z)=\eta(z)T(\f^2 x,\f^2 y,\xi),\quad
                T(x,y,\xi)=\mp T(\f x,\f y,\xi).
\]
\end{proposition}

Then the tensors $\frac{1}{2}\left\{T-L_{2,0}(T)\right\}$ and
$\frac{1}{2}\left\{T+L_{2,0}(T)\right\}$ are the projections of
$\W_2$ in $\W_{2,1}$ and $\W_{2,2}$, respectively. Moreover, we
have $p_{2,j}\circ p_{2,j}=p_{2,j}$ $(j=1,2)$ and
$p_{2,1}+p_{2,2}=\Id_{\W_{2}}$. Therefore, taking into account
$p_2$, we obtain
\begin{proposition}\label{prop-p_2j}
Let $T\in\T$ and $p_{2,j}$ $(j=1,2)$ be the projection operators
of $\T$ in $\W_{2,j}$, generated by the decomposition above. Then
we have
\[
\begin{array}{l}
p_{2,1/2,2}(T)(x,y,z)=\frac{1}{2}\eta(z)\left\{T(\f^2 x,\f^2
y,\xi)\mp T(\f x,\f y,\xi)\right\}.
\end{array}
\]
\end{proposition}

According to \corref{cor-t-W1} (ii), \propref{prop-T2k} and
\eqref{t} we obtain the following
\begin{corollary}\label{cor-t-W2}
The torsion forms of $T$ are zero in each of the subspaces
$\W_{2,1}$ and  $\W_{2,2}$, \ie
    if $T\in\W_{2,1}\oplus\W_{2,2}$, then $t=t^*=\hat{t}=0$.
\end{corollary}

\subsection{The subspace $\W_3$}
\begin{lemma}
The following operators $L_{3,k}$ $(k=0,1)$ are involutive
isometries on $\W_3$ and invariant with respect to $\G\times\I$$:$
\begin{gather}
\begin{array}{l}
L_{3,0}(T)(x,y,z)=\left\{\eta(x)T(\xi,\f y,\f
z)\right\}_{[x\leftrightarrow y]},\nonumber\\[4pt]
L_{3,1}(T)(x,y,z)=\left\{\eta(x)T(\xi,\f^2 z,\f^2
y)\right\}_{[x\leftrightarrow y]}. \nonumber
\end{array}
\end{gather}
\end{lemma}

By virtue of their action, we obtain consecutively the
corresponding invariant and orthogonal eigenspaces:
\begin{gather}
\begin{array}{rl}
\W_3^-=\left\{T\in \W_3\ \vert\ L_{3,0}(T)=-T\right\},\quad
&\W_3^+=\left\{T\in \W_3\ \vert\ L_{3,0}(T)=T\right\},\nonumber
\\[4pt]
\W_{3,1/3,2}=\left\{T\in \W_3^-\ \vert\ L_{3,1}(T)=\pm
T\right\},\quad &\W_{3,3/3,4}=\left\{T\in \W_3^+\ \vert\
L_{3,1}(T)=\pm T\right\}.\nonumber
\end{array}
\end{gather}

In such a way, we get
\begin{theorem}\label{thm-T3k}
The decomposition $
\W_3=\W_{3,1}\oplus\W_{3,2}\oplus\W_{3,3}\oplus\W_{3,4} $ is
orthogonal and in\-va\-riant with respect to the structure group.
\end{theorem}

\begin{proposition}\label{prop-T3k}
The subspaces of $\W_3$ are determined by$:$
\[
\begin{split}
\W_{3,1/3,2}:\quad &T(x,y,z)=\left\{\eta(x)T(\xi,\f^2 y,\f^2
z)\right\}_{[x\leftrightarrow y]},\quad\\[4pt]
                &T(\xi,y,z)=\pm T(\xi,z,y)=-T(\xi,\f y,\f z); \\[4pt]
\W_{3,3/3,4}:\quad &T(x,y,z)=\left\{\eta(x)T(\xi,\f^2 y,\f^2
z)\right\}_{[x\leftrightarrow y]},\\[4pt]
                &T(\xi,y,z)=\pm T(\xi,z,y)=T(\xi,\f y,\f z).
\end{split}
\]
\end{proposition}

By virtue of \corref{cor-t-W1} (iii), \propref{prop-T3k} and
\eqref{t} we obtain
\begin{corollary}\label{cor-t-T3k}
The torsion forms $t$ and $t^*$ of $T$ 
are zero in $\W_{3,k}\subset\W_{3}$ $(k=2,3,4)$.
\end{corollary}

Let us remark that $\W_{3,1}$ can be additionally decomposed to
three subspaces determined by conditions $t=0$, $t^*=0$ and
$t=t^*=0$, respectively, \ie
\begin{equation*}\label{W31}
\W_{3,1}=\W_{3,1,1}\oplus \W_{3,1,2}\oplus \W_{3,1,3},
\end{equation*}
where
\begin{equation*}\label{W311W312W313}
\begin{array}{c}
\W_{3,1,1}=\left\{T\in \W_{3,1}\ \vert\ t\neq 0,\
t^*=0\right\},\qquad
\W_{3,1,2}=\left\{T\in \W_{3,1}\ \vert\ t=0,\ t^*\neq 0\right\}, \\[4pt]
\W_{3,1,3}=\left\{T\in \W_{3,1}\ \vert\ t=0,\ t^*=0\right\}.
\end{array}
\end{equation*}

\begin{proposition}\label{prop-p_3k}
Let $T\in\T$ and $p_{3,k}$ $(k=1,2,3,4)$ be the projection
operators of $\T$ in $\W_{3,k}$, generated by the decomposition
above. Then we have
\[
\begin{array}{l}
p_{3,k}(T)(x,y,z)=
                \frac{1}{4}\left\{\eta(x)A_{3,k}(y,z)
-\eta(y)A_{3,k}(x,z)\right\},
\end{array}
\]
where
\[
\begin{array}{c}
A_{3,1/3,2}(y,z)=T(\xi,\f^2 y,\f^2 z)
                \pm T(\xi,\f^2 z,\f^2 y)
                -T(\xi,\f y,\f z)
                \mp T(\xi,\f z,\f y),
\\[4pt]
                A_{3,3/3,4}(y,z)=T(\xi,\f^2 y,\f^2 z)
                \pm T(\xi,\f^2 z,\f^2 y)
                +T(\xi,\f y,\f z)
                \pm T(\xi,\f z,\f y).
\end{array}
\]
\end{proposition}

\subsection{The subspace $\W_4$}
Finally, we only denote $\W_4$ as $\W_{4,1}$ and it is determined
as follows
\begin{equation*}
\W_{4,1}:\quad
T(x,y,z)=\eta(z)\left\{\eta(y)\hat{t}(x)-\eta(x)\hat{t}(y)\right\}.
\end{equation*}
Obviously, the projection operator $p_{4,1}: \T
\rightarrow\W_{4,1}$ has the form
\begin{equation}\label{p_41}
p_{4,1}(T)(x,y,z)=\eta(z)\left\{\eta(y)\hat{t}(x)-\eta(x)\hat{t}(y)\right\}.
\end{equation}

\subsection{The fifteen subspaces of $\T$}
In conclusion of the decomposition explained above, we combine
Theorems \ref{thm-W1234}, \ref{thm-T1k}, \ref{thm-T2k} and
\ref{thm-T3k}. We denote the subspaces $\W_{i,j}$ and $\W_{i,j,k}$
by $\T_{s}, s\in\{1,2,\dots,15\}$ as follows:
\begin{equation}\label{Ti}
\begin{array}{lllll}
\T_{1}=\W_{1,1,1},\quad &\T_{2}=\W_{1,1,2},\quad
&\T_{3}=\W_{1,2},\quad &\T_{4}=\W_{1,3,1},\quad
&\T_{5}=\W_{1,3,2},\\[4pt]%
\T_{6}=\W_{1,4},\quad &\T_{7}=\W_{2,1},\quad
&\T_{8}=\W_{2,2},\quad &\T_{9}=\W_{3,1,1},\quad &\T_{10}=\W_{3,1,2},\\[4pt]%
\T_{11}=\W_{3,1,3},\quad &\T_{12}=\W_{3,2},\quad
&\T_{13}=\W_{3,3},\quad &\T_{14}=\W_{3,4},\quad &\T_{15}=\W_{4,1}.
\end{array}
\end{equation}
We obtain the following main statement in the present paper
\begin{theorem}\label{thm-15podpr}
Let $\T$ is the vector space of the torsion tensors of type (0,3)
over the vector space $V$ with almost contact $B$-matric structure
$(\f,\xi,\eta,g)$. The decomposition
\begin{equation}\label{TT-15podpr}
\T=\T_{1}\oplus\T_{2}\oplus\cdots\oplus\T_{15}
\end{equation}
is orthogonal and invariant with respect to the structure group
$\G\times\I$.
\end{theorem}

In the following section we discuss three known natural
connections with torsion on $\M$. Natural connections are a
generalization of the Levi-Civita connection.

\section{Known Natural Connections in the Introduced Classification}\label{sec:3}

Let $\M$ be an almost contact B-metric manifold. The tangent space
$T_pM$ at an arbitrary point $p$ in $M$ is a vector space equipped
with an almost contact B-metric structure.

It is well known that any metric connection $D$ (i.e. $Dg=0$) is
completely determined by its torsion tensor $T$ with
\begin{equation}\label{Hay}
2g\left(D_xy-\n_xy,z\right)=T(x,y,z)-T(y,z,x)+T(z,x,y).
\end{equation}
Then the subspace $\T_{s}$ $(s=1,2,\dots,15)$, where $T$ belongs,
is an important characteristic of $D$. In such a way the
conditions for $T$ described as the subspace $\T_{s}$ give rise to
the corresponding class of the connection with respect to its
torsion tensor.


A metric connection $D$ is called a \emph{natural connection} on
$\M$ if the almost contact structure $(\f,\xi,\eta)$ as well as
the B-metric $g$ (consequently also $\tg$) are parallel regarding
it, \ie $D\f=D\xi=D\eta=Dg=D\tg=0$. Therefore, an arbitrary
natural connection $D$ on $\M\notin\F_0$ plays the same role like
$\n$ on $\M\in\F_0$. Obviously, $D$ and $\n$ coincide when
$\M\in\F_0$. Because of that, we are interested in natural
connections on $\M\notin\F_0$.

\begin{theorem}\label{thm-nat}
A linear connection $D$ is natural on $\M$ if and only if
$D\f=Dg=0$.
\end{theorem}
\begin{proof}
It is known, that a linear connection $D$ is a natural connection
on $(M,\f,\xi,\allowbreak\eta,g)$ if and only if
the following properties for $Q(x,y,z)=g\left(D_xy-\n_xy,z\right)$ are valid \cite{Man31}$:$%
\begin{gather}
 Q(x,y,\f z)-Q(x,\f y,z)=F(x,y,z),
\quad
 Q(x,y,z)=-Q(x,z,y).\label{1ab}
\end{gather}
These conditions are equivalent to $D\f=0$ and $Dg=0$,
respectively. Moreover, $D\xi=0$ is equivalent to the relation
$Q(x,\xi,z)=-F(x,\xi,\f z)$, which is  a consequence of the former
equality of \eqref{1ab}. Finally, since
$\eta(\cdot)=g(\cdot,\xi)$, then supposing $Dg=0$ we have $D\xi=0$
if and only if $D\eta=0$. Thus, the statement is truthful.
\end{proof}

\begin{proposition}\label{prop-nat}
Let $D$ be a natural connection with torsion $T$ on an almost
contact B-met\-ric manifold $M$. Then the following implications
hold true:
\[
\begin{array}{c}
T\in\T_{1}\oplus\T_{2}\oplus\T_{6}\oplus\T_{12}  \Rightarrow  M\in\F_0;\\[4pt]
\begin{array}{llllll}
T\in\T_{3} \Rightarrow   M\in\F_3; \qquad & %
T\in\T_{4} \Rightarrow   M\in\F_1; \qquad & %
T\in\T_{5} \Rightarrow   M\in\F_2; \\[4pt] %
T\in\T_{7} \Rightarrow   M\in\F_7; \qquad &%
T\in\T_{8}  \Rightarrow  M\in\F_8\oplus\F_{10};\qquad &%
T\in\T_{9} \Rightarrow   M\in\F_5; \\[4pt]%
T\in\T_{10} \Rightarrow   M\in\F_4; \qquad &%
T\in\T_{11} \Rightarrow   M\in\F_6; \qquad &%
T\in\T_{13} \Rightarrow   M\in\F_9; \\[4pt]%
T\in\T_{14} \Rightarrow   M\in\F_{10}; \qquad & %
T\in\T_{15} \Rightarrow   M\in\F_{11}. \qquad &  
\end{array}
\end{array}
\]
\end{proposition}
\begin{proof}
The implications follow from \eqref{Hay}, \eqref{1ab}, \eqref{Fi},
\eqref{Ti} and the corresponding characteristic conditions of
$\W_{i,j}$ and $\W_{i,j,k}$ as well as the projection operators
$p_{i,j}$. We show the proof in detail for some classes and the
rest follow in a similar way.

By virtue of \eqref{Hay} and \eqref{1ab} we have
\begin{equation}\label{FT}
\begin{array}{l}
2F(x,y,z)=T(x,y,\f z)-T(y,\f z,x)+T(\f z,x,y)\\[4pt]
\phantom{2F(x,y,z)}-T(x,\f y,z)+T(\f y,z,x)-T(z,x,\f y).
\end{array}
\end{equation}

Let us consider $T\in\W_{1,1}=\T_1\oplus\T_2$, which is equivalent
to $T=p_{1,1}(T)$. Then, ac\-cording to \propref{prop-p_1i}, we
have
\[
\begin{array}{l}
T(x,y,z)=-\frac{1}{4}\bigl\{T(\f^2 x,\f^2 y,\f^2 z)-T(\f x,\f
y,\f^2 z)\\[4pt]
\phantom{T(x,y,z)=-\frac{1}{4}\bigl\{}-T(\f x,\f^2 y,\f z)-T(\f^2
x,\f y,\f z)\bigr\},
\end{array}
\]
which together with \eqref{FT} imply $F(x,y,z)=0$. Therefore, we
obtain $M\in\F_0$.

Now, let us suppose $T\in\W_{1,2}=\T_3$ and hence $T=p_{1,2}(T)$,
which has the following form, taking into account
\propref{prop-p_1i}:
\[
\begin{array}{l}
T(x,y,z)=-\frac{1}{4}\bigl\{T(\f^2 x,\f^2 y,\f^2 z)-T(\f x,\f
y,\f^2 z)\\[4pt]
\phantom{T(x,y,z)=-\frac{1}{4}\bigl\{}+T(\f x,\f^2 y,\f z)+T(\f^2
x,\f y,\f z)\bigr\}.
\end{array}
\]
Then, according to the latter equality and \eqref{FT}, we obtain
\begin{equation}\label{FTT}
\begin{array}{l}
F(x,y,z)=-\frac{1}{4}\bigl\{-T(\f^2 x,\f^2 y,\f z)+T(\f x,\f
y,\f z)+T(\f^2 x,\f y,\f^2 z)\\[4pt]
\phantom{F(x,y,z)=-\frac{1}{4}\bigl\{}%
+T(\f x,\f^2 y,\f^2 z)-T(\f z,\f^2 x,\f^2 y)-T(\f^2 z,\f x,\f^2 y)\\[4pt]
\phantom{F(x,y,z)=-\frac{1}{4}\bigl\{}%
 +T(\f^2 z,\f^2 x,\f y)-T(\f z,\f x,\f y)\bigr\}
\end{array}
\end{equation}
and consequently $F(\xi,y,z)=F(x,y,\xi)=0$. Next, we take the
cyclic sum of \eqref{FTT} by the arguments $x,y,z$ and the result
is $\sx F(x,y,z)=0$. Therefore, $M$ belongs to $\F_3$.
\end{proof}

Bearing in mind the class of almost contact B-metric manifolds
with $N=0$ and \propref{prop-nat}, we obtain immediately
\begin{corollary}
An almost contact B-metric manifold $M=\M\in\F_i\setminus\F_0$ is
normal, \ie $N=0$, if the torsion of an arbitrary natural
connection on $M$ belongs to
$\T_{4}\oplus\T_{5}\oplus\T_{9}\oplus\T_{10}\oplus\T_{11}$.
\end{corollary}

Similarly, \propref{prop-Nhat=0} and \propref{prop-nat} imply
\begin{corollary}
An almost contact B-metric manifold $M=\M\in\F_i\setminus\F_0$ has
$\widehat{N}=0$, if the torsion of an arbitrary natural connection
on $M$ belongs to $\T_{3}\oplus\T_{7}$.
\end{corollary}

\subsection{The $\f$B-connection in the classification}

In \cite{ManGri2}, it is introduced a natural con\-nect\-ion $\dot{D}$ on $\M$ in any basic class 
by
\[
\begin{array}{l}
\dot{D}_xy=\n_xy+\frac{1}{2}\bigl\{\left(\n_x\f\right)\f
y+\left(\n_x\eta\right)y\cdot\xi\bigr\}-\eta(y)\n_x\xi.
\end{array}
\]
In \cite{ManIv37}, this connection is called a
\emph{$\f$B-connection}. It is studied for some classes of the
manifolds $\M$ in \cite{ManGri2,Man3,Man4,ManIv37}. The
$\f$B-connection is the odd-dimensional analogue of the
B-connection on the corresponding almost complex manifold with
Norden metric, studied  in \cite{GaGrMi} for the class of the
conformal K\"ahler manifold with Norden metric.

This connection has a torsion tensor and torsion 1-forms as
follows:
\begin{gather}
\begin{array}{l}
\dot{T}(x,y,z)=\bigl\{-\frac{1}{2}F(x,\f y,\f^2 z)+\eta(x)F(y,\f z,\xi)
+\eta(z)F(x,\f y,\xi)\bigr\}_{[x\leftrightarrow y]}, \label{TD}
\end{array}
\\[4pt]
\begin{array}{c}
    \dot{t}=\frac{1}{2}\left\{\ta^*+\ta^*(\xi)\eta\right\},\quad
    \dot{t}^{*}=-\frac{1}{2}\left\{\ta+\ta(\xi)\eta\right\},\quad 
    \hat{\dot{t}}=-\om\circ\f. \label{tB} 
\end{array}
\end{gather}

Applying Propositions \ref{prop-p_1i}, \ref{prop-p_2j},
\ref{prop-p_3k} and equation  \eqref{p_41} for the torsion tensor
$\dot{T}$ from \eqref{TD}, we obtain the components of $\dot{T}$
in each of the subspaces $\W_{i,j}$:
\begin{equation}\label{pijT-B}
\begin{array}{rl}
&p_{1,1}(\dot{T})(x,y,z)=0,\quad\\[4pt]
&p_{1,2}(\dot{T})(x,y,z)=-\frac{1}{4}\left\{F(\f^2 x,\f^2 y,\f z)
-F(\f x,\f y,\f z)
\right\}_{[x\leftrightarrow y]},\\[4pt]
&p_{1,3/1,4}(\dot{T})(x,y,z)=-\frac{1}{8}\bigl\{F(\f^2 z,\f^2 y,\f
x)
\pm F(\f^2 x,\f^2 y,\f z)\\[4pt]%
&\phantom{p_{1,3/1,4}(\dot{T})(x,y,z)=-\frac{1}{8}\bigl\{}%
\pm F(\f x,\f y,\f z)
\bigr\}_{[x\leftrightarrow y]},\\[4pt]
&p_{2,1/2,2}(\dot{T})(x,y,z)=-\frac{1}{2}\eta(z)\left\{F(\f^2 x,\f
y,\xi)\mp F(\f x, y,\xi)
\right\}_{[x\leftrightarrow y]},\\[4pt]
&p_{3,1/3,2}(\dot{T})(x,y,z)=
                \frac{1}{4}\bigl\{\eta(y)\left[F(\f^2 x,\f z,\xi)\pm F(\f^2 z,\f x,\xi)
                \right.\\[4pt]
                &\phantom{p_{3,1/3,2}(\dot{T})(x,y,z)=
                \frac{1}{4}\bigl\{\eta(y)\left[\right.}\left.
                -F(\f x, z,\xi)\mp F(\f z, x,\xi)\right]\bigr\}_{[x\leftrightarrow y]},\\[4pt]
&p_{3,3}(\dot{T})(x,y,z)=
                \frac{1}{4}\bigl\{\eta(y)\left[F(\f^2 x,\f z,\xi)+F(\f^2 z,\f x,\xi)
\right.\\[4pt]
                &\phantom{p_{3,3}(\dot{T})(x,y,z)=
                \frac{1}{4}\bigl\{\eta(y)\left[\right.}\left.
                +F(\f x, z,\xi)+F(\f z, x,\xi)\right]\bigr\}_{[x\leftrightarrow y]},\\[4pt]
%
&p_{3,4}(\dot{T})(x,y,z)=\frac{1}{4}\bigl\{\eta(y)\left[F(\f^2
x,\f z,\xi)-F(\f^2 z,\f x,\xi)+F(\f x, z,\xi)
\right.\\[4pt]
                &\phantom{p_{34}(\dot{T})(x,y,z)=
                \frac{1}{4}\bigl\{\eta(y)\left[\right.}\left.
                -F(\f z, x,\xi)
                +2F(\xi,\f x,\f^2 z)\right]\bigr\}_{[x\leftrightarrow y]},\\[4pt]
&p_{4,1}(\dot{T})(x,y,z)=\,\eta(z)\left\{\eta(x)\omega(\f
y)-\eta(y)\omega(\f x)\right\}.
\end{array}
\end{equation}

Such a way we establish the position of the torsion of $\dot{D}$
in the classification \eqref{TT-15podpr} as follows 
%
\begin{proposition}\label{prop-fB}
The torsion $\dot{T}$ of the $\f$B-connection on $\M$ belongs to
\(
\T_{3}\oplus\T_{4}\oplus\cdots\oplus\T_{15}\). 
\end{proposition}

\subsection{The $\f$KT-connection in the classification}

In \cite{Man31}, it is introduced a natural connection on $\M$,
called a \emph{$\f$KT-connec\-tion}, which torsion tensor
$\ddot{T}$ is totally skew-symmetric, \ie a 3-form. The
$\f$KT-connection is the odd-dimensional analogue of the
KT-connection introduced in \cite{Mek-08} on the corresponding
class of quasi-K\"ahler almost complex manifolds with Norden
metric.

\begin{corollary}
The $\f$KT-connection exists on an almost contact B-metric
manifold $\M$ if and only if the tensor $\widehat{N}$ vanishes on
it.
\end{corollary}
\begin{proof}
It is proved in \cite{Man31} that $\f$KT-connection exists only on
$\M\in\F_3\oplus\F_7$, \ie the class of almost contact B-metric
manifolds, where $\xi$ is a Killing vector field and the cyclic
sum $\s$ of $F$ by three arguments is zero. According to
\propref{prop-Nhat=0}, the class $\F_3\oplus\F_7$ is characterized
by the condition $\widehat{N}=0$ which completes the proof.
\end{proof}

The unique $\f$KT-connection $\ddot{D}$ is determined by
    \[
        \begin{array}{l}
            g(\ddot{D}_xy,z)=g(\n_xy,z)+\frac{1}{2}\ddot{T}(x,y,z),
        \end{array}
    \]
    where the torsion tensor is defined by
\begin{equation}\label{T37} %
\begin{array}{l}
\ddot{T}(x,y,z)=-\frac{1}{2} \sx\bigl\{F(x,y,\f z)-3\eta(x)F(y,\f
z,\xi)\bigr\}\\[4pt]
\phantom{\ddot{T}(x,y,z)}=\left(\eta\wedge
\D\eta\right)(x,y,z)+\frac{1}{4}\sx
N(x,y,z). %
\end{array}
\end{equation} %
Obviously, the torsion forms of the $\f$KT-connection are zero.

From \eqref{T37}, in a similar way of \eqref{pijT-B}, we get the
following non-zero components of $\ddot{T}$:
\begin{equation}\label{pijT-KT}
\begin{array}{l}
%
p_{1,2}(\ddot{T})(x,y,z)=-\frac{1}{2}\bigl\{F(x,y,\f z)+F(y,z,\f
x)-F(z,x,\f y)\\[4pt]
\phantom{p_{1,2}(\ddot{T})(x,y,z)=-\frac{1}{2}}
-\eta(x)F(y,\f z,\xi)+\eta(y)F(z,\f x,\xi)+\eta(z)F(x,\f y,\xi)\bigr\},\\[4pt]
p_{1,4}(\ddot{T})(x,y,z)=-F(z,x,\f y)-\eta(x)F(y,\f z,\xi),\\[4pt]
p_{2,1}(\ddot{T})(x,y,z)=2\eta(z)F(x,\f y,\xi),\\[4pt]
p_{3,2}(\ddot{T})(x,y,z)=2\eta(x)F(y,\f z,\xi)+2\eta(y)F(z,\f x,\xi).
%
\end{array}
\end{equation}

Therefore we have
\begin{proposition}\label{prop-fKT}
The torsion $\ddot{T}$ of the $\f$KT-connection on $\M\in
\F_3\oplus\F_7$ belongs to \(
\T_{3}\oplus\T_{6}\oplus\T_{7}\oplus\T_{12}\). 
\end{proposition}

\subsection{The $\f$-canonical connection in the classification}

In \cite{ManIv38}, it is introduced a natural connection
$\dddot{D}$ on
 $(M,\f,\xi,\allowbreak\eta,g)$, called a
\emph{$\f$-cano\-nic\-al connection}, if the torsion tensor
$\dddot{T}$ of $\dddot{D}$ satisfies the following identity:
\begin{equation}\label{T-can}
\begin{array}{c}
    \bigl\{\dddot{T}(x,y,z)-\dddot{T}(x,\f y,\f z)
    -\eta(x)\left\{\dddot{T}(\xi,y,z)
    -\dddot{T}(\xi, \f y,\f z)\right\}\\[4pt]
    -\eta(y)\left\{\dddot{T}(x,\xi,z)-\dddot{T}(x,z,\xi)-\eta(x)\dddot{T}(z,\xi,\xi)\right\}\bigr\}_{[y\leftrightarrow
    z]}=0.
\end{array}
\end{equation}

Let us remark that the restriction the $\f$-canonical connection
of $\M$ on the contact distribution $\ker(\eta)$ is the unique
canonical connection of the corresponding almost complex manifold
with Norden metric, studied in \cite{GaMi87}.

The torsion tensor of the the $\f$-canonical connection is
\begin{equation}\label{T-can-F}
\begin{array}{l}
\dddot{T}(x,y,z)=\dot{T}(x,y,z)-\frac{1}{8}\left\{N(\f^2 z,\f^2
y,\f^2 x)+2N(\f z,\f y,\xi)\eta(x)\right\}_{[x\leftrightarrow y]},
\end{array}
\end{equation}
where $\dot{T}$ is the torsion tensor of the $\f$B-connection from
\eqref{TD}. The torsion forms are the same as in \eqref{tB}.

In \cite{ManIv38}, it is proved that the $\f$B-connection and the
$\f$-canonical con\-nec\-tion of the manifold $\M$ coincide if and
only if $N(\f\cdot,\f\cdot)=0$, \ie  on any manifold from $\F_i$,
$i\in\{1,2,\dots,11\}\setminus \{3,7\}$, where the
$\f$KT-connection does not exist. For the rest basic classes,
where the $\f$KT-connection exists, we obtain
\begin{proposition}\label{prop-3D}
Let $\M$ be an arbitrary manifold in $\F_i$, $i\in\{3,7\}$. The
$\f$B-con\-nec\-tion $\dot{D}$ is the average connection of the
$\f$KT-connection $\ddot{D}$ and the $\f$-canonical con\-nec\-tion
$\dddot{D}$, \ie $2\dot{D}=\ddot{D}+\dddot{D}$.
\end{proposition}
\begin{proof}
 By virtue of \eqref{pijT-B}, \eqref{pijT-KT} and
\eqref{T-can-F}
we obtain: \\[4pt]
1) for $\F_3$
\[
\begin{split}
p_{1,2}(\dot{T})(x,y,z)&=p_{1,2}(\ddot{T})(x,y,z)
=p_{1,2}(\dddot{T})(x,y,z)\\[4pt]%
&= -\frac{1}{2}\left\{F(\f^2x,\f^2y,\f z)+F(\f^2y,\f^2z,\f
x)-F(\f^2z,\f^2x,\f y)\right\},\\[4pt]%
2p_{1,4}(\dot{T})(x,y,z)&=
p_{1,4}(\ddot{T})(x,y,z)=-F(\f^2z,\f^2x,\f y),\quad
p_{1,4}(\dddot{T})(x,y,z)=0;
\end{split}
\]
2) for $\F_7$
\[
\begin{split}
p_{2,1}(\dot{T})(x,y,z)&=p_{2,1}(\ddot{T})(x,y,z)
=p_{2,1}(\dddot{T})(x,y,z)= 2\eta(z)F(x,\f y,\xi),\\[4pt]%
2p_{3,2}(\dot{T})(x,y,z)&=
p_{3,2}(\ddot{T})(x,y,z)=2\left\{\eta(x)F(y,\f
z,\xi)-\eta(y)F(x,\f z,\xi)\right\},\\[4pt] p_{3,2}(\dddot{T})(x,y,z)&=0.
\end{split}
\]
Therefore, we establish that $2\dot{T}=\ddot{T}+\dddot{T}$ for
$\F_3$ and $\F_7$. Then, using \eqref{Hay}, we obtain
$2\dot{Q}=\ddot{Q}+\dddot{Q}$ for the corresponding tensors
$\dot{Q}(x,y,z)=g(\dot{D}_xy-\n_xy,z)$,
$\ddot{Q}(x,y,z)=g(\ddot{D}_xy-\n_xy,z)$,
$\dddot{Q}(x,y,z)=g(\dddot{D}_xy-\n_xy,z)$. Therefore, we
have the statement.
\end{proof}

\propref{prop-nat} and \propref{prop-3D} imply
\begin{corollary}
The torsion of the $\f$-canonical connection on $\M$ belongs to
$\T_{3}$ and  $\T_{7}$ if and only if $\M$ belongs to $\F_3$ and
$\F_7$, respectively.
\end{corollary}

\begin{remark}
The implications in \propref{prop-nat} become equivalences for the
$\f$-canonical connection on $\M\in\F_i$,
$i\in\{1,2,\dots,11\}\setminus \{3,7\}$, according to
\cite{ManIv38}.
\end{remark}

\section*{Acknowledgments}
The authors wish to thank Stefan Ivanov for his useful advices
about this work. This work was financially supported by the
Scientific Research Fund, Paisii Hilendarski University of
Plovdiv, Bulgaria and the German Academic Exchange Service (DAAD).

\end{document}